\documentclass{amsart}

\usepackage{amsmath,amsthm,amsfonts,amssymb,bm,graphicx, mathrsfs}
\usepackage{enumerate}
\usepackage{hyperref}
\usepackage{longtable}
\hypersetup{colorlinks=true,citecolor=blue,linkcolor=blue,urlcolor=blue,
pdfstartview=FitH }

\theoremstyle{plain}
\newtheorem{theorem}{Theorem}
\newtheorem{prop}[theorem]{Proposition}
\newtheorem{lemma}[theorem]{Lemma}
\newtheorem{cor}[theorem]{Corollary}
\numberwithin{equation}{section}

\theoremstyle{definition}

\newtheorem*{remark}{Remarks}

\renewcommand{\geq}{\geqslant}
\renewcommand{\leq}{\leqslant}

\newcommand{\halfsum}{\varrho}
\newcommand{\hog}{S}

\DeclareMathOperator{\GL}{GL}
\DeclareMathOperator{\SL}{SL}

\newcommand{\dq}{/\!\!/}
\newcommand{\cptset}{\Omega}
\newcommand{\spectrbound}{\sqrt{5/2}}

\newcommand{\spherical}{\varphi}

\newcommand{\eps}{\varepsilon}

\newcommand{\resid}{\text{res}}
\newcommand{\cuspidal}{\text{cusp}}

\newcommand{\pp}{\text{pp}}
\newcommand{\ac}{\text{ac}}

\newcommand{\boldc}{\mathbf{c}}
\DeclareMathOperator{\spann}{span}

\DeclareMathOperator{\Sp}{Sp}

\DeclareMathOperator{\GSp}{GSp}
\DeclareMathOperator{\SO}{SO}

\DeclareMathOperator{\diag}{diag}

\DeclareMathOperator{\Tr}{Tr}

\DeclareMathOperator{\Ima}{Im}
\DeclareMathOperator{\Rea}{Re}

\DeclareMathOperator{\Ad}{Ad}
\DeclareMathOperator{\ad}{ad}

\newcommand\N{\mathbb{N}}
\newcommand\Q{\mathbb{Q}}
\newcommand\R{\mathbb{R}}
\newcommand\Z{\mathbb{Z}}
\newcommand\C{\mathbb{C}}

\newcommand{\mc}[1]{\mathcal #1}
\newcommand{\mf}[1]{\mathfrak #1}

\newcommand{\wt}{\widetilde}

\newcommand{\sceq}{\mathrel{\mathop:}=}
\newcommand{\seqc}{\mathrel{=\mkern-4.5mu{\mathop:}}}

\newcommand{\mat}[4]{\begin{pmatrix} #1&#2\\#3&#4\end{pmatrix}}

\newcommand{\textmat}[4]{\left(\begin{smallmatrix} #1&#2 \\ #3&#4
\end{smallmatrix}\right)}

\makeatletter
\DeclareRobustCommand\widecheck[1]{{\mathpalette\@widecheck{#1}}}
\def\@widecheck#1#2{%
    \setbox\z@\hbox{\m@th$#1#2$}%
    \setbox\tw@\hbox{\m@th$#1%
       \widehat{%
          \vrule\@width\z@\@height\ht\z@
          \vrule\@height\z@\@width\wd\z@}$}%
    \dp\tw@-\ht\z@
    \@tempdima\ht\z@ \advance\@tempdima2\ht\tw@ \divide\@tempdima\thr@@
    \setbox\tw@\hbox{%
       \raise\@tempdima\hbox{\scalebox{1}[-1]{\lower\@tempdima\box
\tw@}}}%
    {\ooalign{\box\tw@ \cr \box\z@}}}
\makeatother

\begin{document}

\author{Valentin Blomer}
\author{Anke Pohl}
\address{Mathematisches Institut, Bunsenstr. 3-5, 37073 G\"ottingen, Germany} \email{vblomer@math.uni-goettingen.de}  \email{anke.pohl@math.uni-goettingen.de}

 \title[The sup-norm problem on the Siegel modular space of rank two]{The sup-norm problem on the Siegel modular space\\ of rank two}

\thanks{First  author   supported  by the Volkswagen Foundation and a Starting Grant of the European Research Council. Second author supported  by  the Volkswagen Foundation.}


\begin{abstract}   Let $F$ be a square integrable Maa{\ss} form on the Siegel upper half space $\mc H$ of rank $2$ for the Siegel modular group $\Sp_4(\Z)$  with Laplace eigenvalue $\lambda$. If, in addition, $F$ is   
 a joint eigenfunction of the Hecke algebra and $\cptset$ is a compact set in $\Sp_4(\Z)\backslash\mc H$, we show the bound $\| F|_{\cptset} \|_{\infty} \ll_{\cptset} (1+\lambda)^{1-\delta}$ for some global constant $\delta > 0$. \\
 As an auxiliary result of independent interest we prove new uniform bounds for spherical functions on semisimple Lie groups. 
\end{abstract}


\setcounter{tocdepth}{2}  

\maketitle

 \section{Introduction}

Given a Riemannian manifold $X$ of finite volume, it is  natural to ask for various properties of the $L^2$-eigenfunctions of its Laplacian $\Delta$. These eigenfunctions are of importance in a range of fields such as quantum chaos and mathematical physics, harmonic analysis and -- if $X$ has some connection to arithmetic -- number theory. 
 The relation to physics is provided by the fact that the Laplacian coincides, up to scaling, with the Schr\"odinger operator of a freely moving particle on $X$, and hence the $L^2$-eigenfunctions of $\Delta$ are understood as bound states in physics.  Hence typical 
  questions ask about asymptotics (Weyl law), distribution and multiplicities of the eigenvalues of these $L^2$-eigenfunctions, and the distribution of their masses along increasing sequences of eigenvalues.    More refined questions include  the asymptotic behaviour of Wigner distributions or microlocal lifts and their possible quantum limits (quantum ergodicity, quantum unique ergodicity, arithmetic 
quantum unique ergodicity); often entropy bounds  play an important role in this context.  For detailed surveys of recent results see, e.g. \cite{Sarnak_survey, Zelditch_survey}.

In this paper we study the problem of obtaining \emph{pointwise} bounds in the spectral aspect for $L^2$-normalized eigenfunctions on the Siegel modular space of rank $2$,  the quotient space of the Siegel upper half space of rank $2$ by the Siegel modular group $\Sp_4(\Z)$. Such upper bounds for sup-norms provide a measure for the equidistribution of mass of the respective eigenfunction since they estimate the extent to which the eigenfunctions may localize in small sets. The problem of bounding sup-norms is also closely related to the multiplicity problem: if $V_{\lambda}$ denotes the eigenspace of the eigenvalue $\lambda$, then we have the inequality \cite{Sa} 
$$\dim V_{\lambda} \leq \text{vol}(X) \sup_{\substack{\| F \|_2 = 1\\ F \in V_{\lambda}}} \| F \|^2_{\infty}.$$
In particular, good bounds for sup-norms imply good bounds for multiplicities of eigenvalues, and conversely large eigenspaces prevent the possibility to obtain such bounds. 
 A classical example (see e.g.\ \cite{Fa, Iw}) is the sphere $X = S^2$; here all eigenvalues are of the form $\lambda = k(k+1)$ with $k \in \N_0$, and the corresponding eigenspaces are of dimension $2k+1$. One $L^2$-normalized element in $V_{k(k+1)}$ is given by $F(\theta) \sceq  ((2k+1)/4\pi)^{1/2}  p_k(\cos \theta)$ 
 where $\theta$ is the azimuth angle and $p_k$ is the $k$-th Legendre polynomial. In particular, $$\| F \|_{\infty} = \left(\frac{2k+1}{4\pi}\right)^{1/2} \asymp (1+\lambda_F)^{1/4}.$$ 
In general, for any compact Riemannian manifold and $L^2$-normalized Laplace eigenfunctions $F$, one has the bound \cite{Seeger_Sogge, Hoermander_bound, Sogge_Zelditch}
\begin{equation}\label{hoer_bound}
\| F \|_{\infty} \ll (1 + \lambda_F)^{(\dim X - 1)/4}, 
\end{equation}
and the example of the $n$-sphere shows that this bound is sharp in general. 
%As well-known, for any $n$-sphere this bound is sharp due to having large eigenspaces \cite{Fa, Iw}. Thus, for compact Riemannian manifolds in this generality, the bound~\eqref{hoer_bound} cannot be uniformly improved. 
However, if $X$ carries additional symmetries, then it is reasonable to expect significantly stronger estimates.  An important instance of the effect of additional symmetries has been established by Sarnak \cite{Sa}: If $X$ is  a compact Riemannian locally symmetric space of rank $r$, then
\begin{equation}\label{generic}
 \| F \|_{\infty} \ll (1 + \lambda_F)^{(\dim X - r)/4}
 \end{equation}
for joint eigenfunctions $F$ with Laplace eigenvalue $\lambda_F$ of the algebra of all differential operators that are invariant under the Riemannian isometry group of $X$ \cite{Sa}. For non-compact spaces, things are more complicated, and the lower bound for congruence quotients of ${\rm PGL}_n(\R)$ in \cite[Theorem 1.2, Corollary 1.3]{BT} shows that \eqref{generic} cannot serve as a global generic bound for non-compact locally symmetric spaces $X$, at least if the rank is   large. However, when restricted to a fixed bounded subset $\cptset \subseteq X$, we have the corresponding generic bound
\begin{equation}\label{generic_nc}
  \| F\vert_{\cptset} \|_{\infty} \ll_{\cptset} (1 + \lambda_F)^{(\dim X - r)/4},
\end{equation}
 see \cite[(2.4)]{BT}, even in the slightly stronger form $\| F\vert_{\cptset} \|_{\infty} \ll_{\cptset} (1 + \lambda_F)^{(\dim X - r)/4} \| F\vert_{\cptset}\|_2$ if $\cptset$ has non-empty interior.

Many classical examples of Riemannian locally symmetric spaces (in particular, all irreducible Riemannian symmetric spaces of higher rank) are arithmetic and hence enjoy additional symmetries given by the Hecke operators, a commutative family of normal operators. For the joint $L^2$-eigenfunctions of the Hecke algebra and the invariant differential operators, multiplicity one results are known.   Hence in absence of obvious obstructions one might hope to be able to further improve the bound \eqref{generic} or \eqref{generic_nc} for these eigenfunctions.

The archetypical result of a power saving relative to the generic bound \eqref{generic} is due to Iwaniec and Sarnak \cite{IS, Sa} in the situation $X = \Gamma\backslash \Bbb{H}$ where $\Bbb{H}$ is the hyperbolic plane and $\Gamma \leq {\rm SL}_2(\R)$ is a  cocompact  arithmetic lattice or  ${\rm SL}_2(\Z)$. For $L^2$-normalized Hecke Maa{\ss} cusp forms $F$ they proved the bound $\| F \|_{\infty} \ll (1 + \lambda_F)^{5/24 + \varepsilon}$.

Similar results have been obtained for the Hecke-Laplace eigenfunctions on the sphere and ellipsoids \cite{VdK, BM} and on congruence quotients of hyperbolic 3-space \cite{BHM}. The underlying algebraic groups in these cases are $\SL_2(\R) = \SO_0(2, 1)$, $\SO(3)$ and $ \SL_2(\C) = \SO_0(3, 1)$, all of which have real rank at most $1$.  \\

In this paper we consider for the first time a group of \emph{real rank $2$}, the symplectic group 
\[
G\sceq\Sp_4(\R) \cong \SO_0(3, 2),
\]
and an irreducible (arithmetic) lattice in $G$, namely the Siegel modular group
\[
 \Gamma\sceq\Sp_4(\Z).
\]
The Riemannian symmetric space of non-compact type associated to $G$ is the Siegel upper half space $\mc H$ of rank $2$, and the Siegel modular space
\[
X \sceq \Gamma\backslash \mc H
\]
is a non-compact $6$-dimensional arithmetic Riemannian locally symmetric orbifold with one cusp. On $X$ (or $\mc H$) we consider the ($\Gamma$-invariant) joint eigenfunctions in $L^2(X)$ of the algebra of $G$-invariant differential operators on $\mc H$ and the Hecke algebra associated to $\Gamma$. These eigenfunctions are often called square integrable Hecke Siegel Maa{\ss} wave forms of genus (or degree) $2$ for $\Gamma$, containing the important subspace of Hecke Siegel Maa{\ss} cusp forms for $\Gamma$ which are those forms that decay rapidly towards the cusp of $X$. For shortness we refer to all of these forms just as \textit{joint eigenfunctions} in $L^2(X)$. For precise definitions and more details we refer to Sections~\ref{sec2} -- \ref{sec:SMF} below and standard monographs on Siegel modular forms, e.g.\@ \cite{Fr}. For the restriction of $F$ to any compact set $\cptset \subseteq X$ we prove the following power saving of the bound \eqref{generic_nc}:

\begin{theorem}\label{mainthm}
There exists $\delta>0$ such that for any compact subset $\cptset$ of $\Gamma\backslash\mc H$ and any $L^2$-normalized joint eigenfunction $F$ in $L^2(\Gamma \backslash \mathcal{H})$ with Laplace eigenvalue $\lambda_F$ we have
\[
  \| F\vert_{\cptset} \|_{\infty} \ll_{\cptset} (1+\lambda_{F})^{1-\delta}.
\]
\end{theorem} 

\begin{remark} 
\begin{enumerate}[(1)]
\item  Our proof provides an explicit value for $\delta$ ($> 10^{-6}$), but we have not optimized the numerical value. 
\item An inspection of the proof shows that we get better bounds when the Langlands parameters $ (\mu_1, \mu_2)$ of $F$ approach  the walls of the Weyl chambers. Using  \eqref{cexplicit} directly in \eqref{mainbound} we obtain the bound
$$\ll \big((1+|\mu_1|)(1+|\mu_2|)(1+|\mu_1 + \mu_2|)(1 + |\mu_1 - \mu_2|)\big)^{1/2 - \delta}$$
in the situation of Theorem \ref{mainthm} (for some possibly different $\delta > 0$) which we  slightly simplified as $(1+|\mu_1|^2 +|\mu_2|^2)^{1-\delta}$. This is due to the fact that the Plancherel density drops close to the walls of the Weyl chambers, but on the other hand the behaviour of the spherical functions becomes more complicated as stationary points tend to blow up to non-trivial submanifolds. We deal with this problem in Theorem \ref{sphericalbound} below. 
\item  The statement of Theorem \ref{mainthm} includes the cuspidal and the residual spectrum, and in particular possible exceptional forms $F$ where the Ramanujan conjecture fails to hold. Moreover, our result and proof is  independent of any bounds towards the Ramanujan conjecture at finite places.
\item We do not investigate the behaviour in the cusp and treat only an arbitrary, but fixed compact piece of the manifold. Our principal method, based on a (pre-)trace formula, cannot distinguish between cuspidal and non-cuspidal constituents (in particular Eisenstein series) of the  automorphic spectrum, as they are treated evenly. Therefore  it is \textit{a priori} clear that our sup-norm bounds must deteriorate as we approach the cusp. Of course, cusp forms are rapidly decaying towards the cusps. In the classical situation of genus $1$ this can be quantified rather easily by the Fourier expansion, but it requires some highly non-trivial input such as precise uniform bounds for Bessel functions on the analytic side, and Rankin-Selberg theory combined with a famous and deep result of Hoffstein-Lockhart 
\cite{HL} on values of symmetric square $L$-functions at the edge of the critical strip on the arithmetic side. Such information is not available for genus $2$. As mentioned earlier, the  size of cusp forms  towards the cusps is, in higher rank,
 a subtle issue, and it is not even obvious that \eqref{generic} is true in our situation. 
\item For completeness, it should be mentioned that there are various related problems which have been studied recently in the context of Hecke eigenforms on hyperbolic $2$- and $3$-manifolds: on the one hand one can study \emph{lower} bounds   which may arise   from at least three different sources: (a) generic fluctuations that are slightly stronger than expected from the random wave model \cite{Mi1}; (b) degenerate behaviour of special functions which leads to peaks high in the cusp \cite{Sa, Te, BT}; and (c$ $) embedded submanifolds \cite{RS, Mi2} allowing eigenfunctions that are functorial lifts with non-generic behaviour. On the other hand, one can let the underlying space vary, and study the sup-norm of eigenfunctions on a sequence of covers $X_N$ of $X$ in terms of the volume of $X_N$ \cite{HT3, BM, BHM}. 
\end{enumerate}
\end{remark}

We now describe some of the methods involved in the proof of Theorem \ref{mainthm}. We view this as a prototype of a general higher rank situation, and we remark that the techniques can also be applied to treat related situations such as the case $G = GL(3)$ in a similar fashion. As pioneered in \cite{IS}, the proof starts with an amplified pre-trace formula. We highlight at this point some of the novel ingredients in this paper.

The amplifier we use here is based on not only the standard Hecke operators for the lattice $\Gamma=\Sp_4(\Z)$, but also involves those defined on individual double cosets. For the analysis of the amplifier, quite precise knowledge on their combinatorics in the Hecke algebra is required, which is rather unwieldy, but potentially useful in other situations. We refer to Sections~\ref{sec3} and \ref{secpre} for details and the explicit construction of this amplifier which also implements some of the recent  advances introduced in \cite{Ve} and \cite{BHM}.\\

The geometric side of the pre-trace formula yields a counting problem that reflects the algebraic structure of a maximal compact subgroup of the isometry group of the considered space. In our setup, the isometry group is $G=\Sp_4(\R)$, and any fixed maximal compact subgroup $K$ is isomorphic to $U(2)$ which is topologically (and almost algebraically) isomorphic to $S^1 \times S^3$. As one may therefore presume, the counting problem translates to twisted binary and quaternary quadratic forms. Roughly speaking, we need to bound the number of integral points in a $\delta$-neighbourhood of the intersection of 4 particular quadrics in 8 variables. Our bound, Proposition~\ref{prop1} below, is essentially best possible for sufficiently small $\delta$.  We refer to Section~\ref{sec7}, in particular to \eqref{mat1a} and \eqref{matalt}, for details. The counting techniques developed in this paper work in much greater generality. \\

Finally, we need uniform bounds for the inverse spherical transform of test functions localized at a given point   in $\mf a^*$ (related to a spectral parameter and with large distance to the origin). This brings us to the second main result that we would like to highlight at this point.  Bounds for inverse spherical transforms are often consequences of the decay of the spherical functions $\spherical_\lambda$ on $G/K$ for large parameter $\lambda\in \mf a^*$ (see Sections \ref{sec:bounds} and \ref{sec2} for the notation). Strong  bounds for $\spherical_{\lambda}$ have been obtained in particular by \cite[Theorem 11.1]{DKV} and \cite[Theorem 1.3]{Ma}, but neither of these bounds is sufficiently uniform for our purposes: the bound in \cite{DKV} requires the argument $\exp(H)$ of $\spherical_\lambda$ to stay away from the identity  in  $G\dq K = K\backslash G/K = A/W$ by a fixed amount, and the bound in \cite{Ma} requires the parameter $\lambda$ to stay away from the walls of the Weyl chambers by a fixed 
amount. Although it might be possible 
to remove these assumptions, we proceed differently   to prove the following 
uniform 
bound, which is also of independent interest:

\begin{theorem}\label{sphericalbound} 
Let $G$ be a real connected semisimple Lie group with finite center. Let $B \subseteq \mathfrak{a}$ be a bounded subset.  Then for any   $H \in B $   and  any $\xi = \lambda + i\eta \in \mf a^* + iC_\halfsum$, the elementary spherical function $\spherical_{\xi}$ of $G$ with parameter $\xi$ satisfies
\[
\spherical_{\xi}(\exp(H)) \ll_B \prod_j(1 + \|  \lambda_j \| \cdot \| H_j \|)^{-1/2},
\]
where $\lambda_j$ and $H_j$ denote the projections onto the simple factors of $\mathfrak{a}$ resp.\ $\mathfrak{a}^{\ast}$.  
\end{theorem}

Unlike the bounds in \cite{DKV, Ma}, the exponent is not best possible (we apply the stationary phase method only in one dimension), but it suffices for our application. The novelty here is the fact that the bound is    completely uniform as $\lambda$ approaches $\infty$ and/or walls of Weyl chambers and as $\exp(H)$ approaches the identity. We expect this result to have applications also in other situations. \\ 

\textbf{Acknowledgement:} The authors would like  to thank Jim Arthur, Farrell Brumley, Gergely Harcos and Ralf Schmidt for very helpful comments on various aspects of this work.

\section{Spherical functions}\label{sec:bounds}

\subsection{Spherical functions and spherical transform} 
Let $G$ be an arbitrary real connected semisimple Lie group with finite center, and let $\mf g$ denote its Lie algebra. Fix an Iwasawa decomposition $G = K A N$. Let $\mf a$ denote the Lie algebra of $A$, and let $H\colon G\to \mf a$ be the Iwasawa projection defined by 
\[
 g \in K \exp(H(g)) N \quad\text{for all $g\in G$.}
\]
Let $\mf a^*$ denote the dual of $\mf a$, and $\mf a^*_\C$ its complexification. Let $\Sigma$ be the set of (restricted) roots of $\mf g$, and fix a maximal subset $\Sigma^+$ of positive roots. Let $\halfsum$ denote the half-sum of the positive roots, weighted with multiplicity: 
\[
 \halfsum = \sum_{\alpha\in\Sigma^+} \dim\mf g_\alpha \cdot \alpha.
\]
Here, $\mf g_\alpha$ denotes the restricted root space of $\alpha$. Let $W$ denote the Weyl group of $G$. As usual, let $G\dq K$ denote the double quotient $K\backslash G/K$ and recall that it is isomorphic to $A/W$ by Cartan decomposition. For later purposes we define the Cartan  projection $C \colon G \to \mathfrak{a}/W$ via
\begin{equation}\label{defC}
g = k_1 \exp(C(g)) k_2  \quad\text{for appropriate $k_1, k_2\in K$.}
\end{equation}
The Killing form 
\[
 \langle X, Y\rangle \sceq \Tr(\ad X \circ \ad Y) 
\]
defines an inner product on $\mf a$, which carries over to $\mf a^*$ via root vectors. Thus, the induced inner product on $\mf a^*$ is given by  
$ \langle \lambda,\mu\rangle = \langle H_\lambda, H_\mu\rangle$ 
for $\lambda,\mu\in\mf a^*$, where $H_\lambda\in\mf a$ is determined by $\lambda(H) = \langle H, H_\lambda\rangle$ for all $H\in\mf a$. We denote the corresponding norms by $\|\cdot\|$. We note that the inner product on $\mf a^*$ extends in a straightforward way to a $\C$-bilinear symmetric form on $\mf a_\C^*$, the complexification of $\mf a^*$. We remark that the extension of $\langle\cdot,\cdot\rangle$ is not an inner product on $\mf a^*_\C$. For $\lambda=\nu+i\mu$ with $\nu,\mu\in\mf a^*$ we set $\|\lambda\|^2 \sceq \|\nu\|^2 + \|\mu\|^2$, which coincides with  $\langle\lambda,\lambda\rangle$ for $\lambda \in \mathfrak{a}^{\ast}$, but not for general $\lambda \in \mf a_\C^*$.\\

From Harish-Chandra's monumental work it is known that the spherical functions on $G$ are parametrized by $\mf a^*_\C/W$. For any $\lambda \in \mathfrak{a}_{\C}^*$, the associated spherical function $\spherical_\lambda\colon G\dq K \to\C$ is given by 
\begin{equation}\label{defspher}
\spherical_{\lambda}(g) \sceq \int_K e^{(-\halfsum + i\lambda)H(gk)}dk.
\end{equation}
It is well-known that any two such functions $\spherical_\lambda$ and $\spherical_\mu$ coincide if and only if $\lambda = w.\mu$ for some element $w\in W$, and $\spherical_\lambda$ is bounded (by $1$) if and only if $\lambda \in \mf a^* + i C_\halfsum$ (\cite{Helgason_Johnson}),  where $C_\halfsum$ denotes the convex hull of the points $w.\halfsum$, $w\in W$. 

Let $C^\infty_c(G\dq K)$ denote the space of compactly supported bi-$K$-invariant smooth complex-valued functions. For $f\in C^\infty_c(G\dq K)$, its spherical transform  is defined by 
\begin{equation}\label{sphericaltransform}
\widetilde{f}(\lambda) \sceq  \int_G f(g) \spherical_{-\lambda}(g) dg. 
\end{equation}
We briefly recall the Paley--Wiener Theorem and the Harish-Chandra Inversion Formula for $G$. For more details see e.g.\@ \cite{Gangolli} or \cite[Chap.\@ IV]{He}. For $R>0$ let $\mc H^R(\mf a^*_\C)$ denote the space of entire functions $f\colon \mf a^*_\C\to\C$ that satisfy 
\[
 f(\lambda)  \ll_N (1 + \|\lambda\|)^{-N} e^{R |\Ima \lambda|} \quad \text{for all $\lambda\in\mf a^*_\C$}
\]
for each $N\in\N_0$. Let $\mc H_W^R(\mf a^*_\C)$ denote the subspace of $W$-invariant functions in $\mc H^R(\mf a^*_\C)$. Then 
\[
 \mc H_W(\mf a^*_\C) \sceq \bigcup_{R>0} \mc H^R_W(\mf a^*_\C)
\]
is the space of Paley--Wiener functions. The inversion formula invokes the Harish-Chandra $\boldc$-function, which is the meromorphic function $\boldc\colon \mf a^*_\C\to\C$ given by the  Gindikin-Karpelevich product formula
\[
 \boldc(\lambda) = c_0 \cdot\prod_{\alpha\in \Sigma^+} \boldc_\alpha(\lambda).
\]
For any $\alpha\in \Sigma^+$, the map $\boldc_\alpha$ is given by  
\begin{equation}\label{GK}
 \boldc_\alpha(\lambda) \sceq \frac{2^{-i\lambda_\alpha}\Gamma(i\lambda_\alpha)}{\Gamma\left( \frac{i\lambda_{\alpha}}{2} + \frac{m_\alpha}{4} + \frac12\right) \Gamma\left( \frac{i\lambda_\alpha}{2} + \frac{m_\alpha}{4}\right)},\quad  \lambda_\alpha \sceq \frac{\langle \lambda,\alpha\rangle}{\langle \alpha,\alpha\rangle},
\end{equation}
where $m_\alpha =   \dim\mf g_\alpha$. The constant $c_0$ is determined by $\boldc(-i\varrho) = 1$. 

We endow $A/W \cong G\dq K$ with isometric metrics using the identification $\exp\colon \mf a \to A$. We let $\overline{B_R^A(o)}$ denote the closed ball in $A$ with radius $R$ about the identity $o \in A/W$. Note that the walls of Weyl chambers are null sets with respect to $d\lambda/|\boldc(\lambda)|^2$.

\begin{prop}[Paley--Wiener Theorem and Harish-Chandra Inversion Formula]\label{paley}
The spherical transform \eqref{sphericaltransform} is a bijection of $C^\infty_c(G\dq K)$ onto $\mc H_W(\mf a^*_\C)$. 
For each $R>0$, it restricts to a bijection of the space of functions in $C^\infty_c(G\dq K)$ with support in $\overline{B_R^A(o)}$ onto $\mc H^R_W(\mf a^*_\C)$. 
Its inverse is given by 
\begin{align*}
 f(g) &= \frac{1}{|W|} \int_{\mf a^*} \wt f(\lambda) \spherical_\lambda(g) \frac{d\lambda}{|\boldc(\lambda)|^2}.
\end{align*}
\end{prop}

We see that bounds for the inverse spherical transform $f$ can be obtained from bounds of $\varphi_{\lambda}$. We now prepare for the proof of Theorem \ref{sphericalbound}. We let the elements of the real universal enveloping algebra $U(\mf k)$ act as differential operators on $K$. Thus, if $f\colon K\to \C$ is a smooth function and $Y=X_1\ldots X_d$ ($X_i\in\mf k$) is a homogeneous tensor of degree $d$ over $\mf k$, then 
\[
 f(k;Y) = \frac{d^r}{dt_1\ldots dt_d}\Big\vert_{t_1=\cdots=t_d=0} f( k\exp(t_1X_1)\cdots \exp(t_dX_d)).
\]

For the proof of Theorem~\ref{sphericalbound} we take advantage of the following weak van der Corput lemma, which is easily proven by applying \cite[Proposition~5, pp.~342-343]{St}.

\begin{lemma}\label{vdCorput}
Let $K$ be a compact Lie group with Lie algebra $\mf k$ and Haar measure $dk$. Let $\psi\colon K \to \C$ be a smooth function and $\phi\colon K \to \R$ be smooth and real-valued. Suppose that there exists a (finite) open cover $(U_j)_{j=1}^m$ of $K$, a constant $c>0$ and a constant $d\in\N$ such that for each $j\in\{1,\ldots, m\}$ there exists a homogeneous tensor $Y_j$ over $\mf k$ of degree at most $d$ with $|\phi(k;Y_j)|\geq c$ for all $k\in U_j$. Suppose further that there exists a constant $N > 0$ such that the $C^{d+1}$-norm of $\phi$ and the $C^1$-norm of $\psi$ are bounded by $N$. 
Then 
\[
 \left| \int_K \psi(k) e^{it\phi(k)} dk \right| \ll t^{-\frac1d} \quad (t \rightarrow \infty).
\]
The implied constant  depends on $\phi$ and $\psi$ and only through $c, d, N$. 
\end{lemma}

For the proof of Theorem~\ref{sphericalbound} we consider the integral in \eqref{defspher} as an oscillatory integral, where
the directions of the argument $a = \exp(H) \in A/W = G\dq K$ and $  \lambda$ are fixed and their length are the variables. The phase function is obviously linear in the length of $ \lambda$, and we will linearize it in the length of $H$.  

\subsection{Proof of Theorem \ref{sphericalbound}}
As recalled above, $\|\spherical_{\xi}\|_\infty \leq 1$ for all $\xi \in \mf a^* + i C_\halfsum$. 

Without loss of generality we may assume that $G$ is simple, so that it remains  to prove the bound $\phi_{\xi}(\exp(H)) \ll (\|\lambda \| \cdot \| H \|)^{-1/2}$ for $H$ in a bounded set $B$. We first observe that \cite[Theorem 11.1]{DKV} yields a stronger bound as long as $H$ stays away from the identity. We can therefore assume without loss of generality that $\| H \| \leq C_0$ for a sufficiently small constant $C_0$. We write $a = a_{\delta} = \exp(\delta X) \in A$ for some $0 < \delta < C_0$ and $\| X \| = 1$ and $\lambda = \lambda_{\tau} = \tau \langle \cdot, Y\rangle \in \mathfrak{a}^{\ast}$  for $\| Y \| = 1$. As $a$ and $\lambda$ vary, $X$ and $Y$ vary in a compact set. 
%We fix $X, Y\in \mf a^*$ with $\|X\| = 1 = \|Y\|$. For $0 < \delta < C_0$ set
%\[
% a_\delta \sceq \exp(\delta X) \ \in \mf a,
%\]
%and for $\tau\in\R$ define $\lambda_\tau\in \mf a^*$ by 
%\[
% \lambda_\tau(U) \sceq \tau\langle U, Y\rangle \quad\quad (U\in\mf a).
%\]
%Further, for any $a\in A$ 
For $a \in A$ we set
\[
 F_{a,Y}\colon K \to \R,\quad F_{a,Y}(k) \sceq \langle H(ak), Y\rangle,
\]
and for  $\eta\in\mf a^*$ we let 
\[
 \psi_\eta\colon A\times K \to \R,\quad  \psi_\eta(a,k) \sceq e^{(-\halfsum - \eta)H(ak)}.
\]
As long as $\eta$ is fixed (e.g.\ contained in $C_{\varrho}$) and  $a$ is bounded, $\psi$ and its derivatives are uniformly bounded. With this notation we have 
 \begin{equation}\label{notation}
 \spherical_{\lambda_\tau + i \eta}(a_\delta) = \int_K \psi_{\eta}(a_\delta, k) e^{i\tau F_{a_\delta, Y}(k)} dk.
\end{equation}
To linearize the phase function $F_{a_\delta, Y}$ in the variable $\delta$, let $f_{X,Y}\colon K \to \R$ be defined as
\[
 f_{X,Y}(k) \sceq \langle X, \Ad_k(Y)\rangle. 
\]
By \cite[Proposition~5.5]{DKV}, 
\[
 \frac{d}{d\delta}\Big\vert_{\delta = 0} F_{\exp(\delta X), Y}(k) = f_{X,Y}(k),
\]
and hence
\[
 F_{a_\delta, Y}(k) = \delta  f_{X,Y}(k) + O(\delta^2).
\]
For the derivative and the Hessian of $F_{a_\delta, Y}$ this implies
\[
 F_{a_\delta, Y}(k;U) = \delta f_{X,Y}(k;U) + O(\delta^2),\quad F_{a_\delta, Y}(k;UV) = \delta f_{X,Y}(k;UV) + O(\delta^2)
\]
for all $U,V\in \mf k$.  By \cite[Proposition~1.2]{DKV}, the critical set of $f_{X,Y}$ is $C_{X,Y}\sceq K_XWK_Y$, where $K_X$ denotes the centralizer of $X$ in $K$. The corresponding subset of the Lie algebra $\mf k$ of $K$ is
\[
 \mf c_{X,Y} = \bigcup_{w\in W} \mf k_X + \mf k_{w.Y}, 
\]
where $\mf k_S = \{ Z\in\mf k \mid [S,Z] = 0\}$ for $S\in\mf a$. We claim that $\mf c_{X,Y}\not= \mf k$. Then the critical manifold is not all of $K$, and if   $D_{X,Y}$ denotes the subset of $K$ on which the Hessian of $f_{X,Y}$ vanishes,  \cite[Corollary~1.5]{DKV} yields
\begin{equation}\label{non-deg}
C_{X,Y} \cap D_{X,Y} = \emptyset. 
\end{equation}

To prove the claim $\mf c_{X,Y}\not= \mf k$, it suffices to show that $\mf k_X + \mf k_Y \not=\mf k$ for $\|X\|=1=\|Y\|$.  For $\alpha\in\Sigma^+$ set $\mf k_\alpha \sceq (\mf g_\alpha + \mf g_{-\alpha}) \cap \mf k$ and $\mf k_0 \sceq \mf g_0 \cap \mf k$. Then 
\[
 \mf k_S = \mf k_0 \oplus \bigoplus_{\alpha\in\Sigma^+, \alpha(S) = 0} \mf k_\alpha.
\]
Note that $\mf k_\alpha \not= \{0\}$ for any $\alpha\in\Sigma^+$. Let $\Pi$ denote the set of simple (restricted) roots, and let $\Pi_X$ denote the set of simple roots vanishing on $X$, and $\Pi_Y$ the set of simple roots vanishing on $Y$. Note that $|\Pi_X| \leq |\Pi|-1$ and $|\Pi_Y|\leq |\Pi|-1$. From the classification of the possible root systems (see, e.\,g., \cite[Appendix C]{Knapp_green}) one easily sees that there exists $\alpha\in\Sigma^+$ which is neither in the linear span of $\Pi_X$ nor of $\Pi_Y$.  Thus, $\alpha(X) \not=0\not=\alpha(Y)$ and hence  $\mf k_\alpha \not\subseteq \mf k_X + \mf k_Y$. 
This proves the claim and in particular \eqref{non-deg}.

Now let $U_C, U_D$ be disjoint open neighbourhoods in $K$ of $C_{X,Y}, D_{X,Y}$, respectively. Fix a basis $V_1,\ldots, V_m$ of $\mf k$ (here, $m=\dim\mf k$). Then, for some constant $C_1>0$, 
\begin{align*}
\max_{1 \leq i \leq m} \left| f_{X,Y}(k;V_i)\right| &\geq C_1 \qquad\text{for all $k\in U_C$,}
\\
\max_{1 \leq i, j \leq m} \left| f_{X,Y}(k; V_iV_j) \right| & \geq C_1 \qquad\text{for all $k\in U_D$.}
\end{align*}
By choosing $C_0$ sufficiently small, it follows that 
\begin{align*}
\max_{1 \leq i \leq m} \left| F_{a_\delta,Y}(k;V_i)\right| &\geq \delta C_1/2 \qquad\text{for all $k\in U_C$,}
\\
\max_{1 \leq i, j \leq m} \left| F_{a_{\delta},Y}(k; V_iV_j) \right| & \geq \delta C_1/2 \qquad\text{for all $k\in U_D$.}
\end{align*}
We rewrite \eqref{notation} as 
\begin{align*}
 \varphi_{\lambda_\tau+i\eta}(a_\delta) = \int_K \psi_\eta(a_\delta,k) \exp\Bigl(i\tau\delta (\delta^{-1}F_{a_\delta,Y}(k))\Bigr) dk. 
\end{align*}
By the above,  $K$ can be covered by finitely many open subsets on each of which  a first or second order partial derivative of $\delta^{-1}F_{a_\delta, Y}$ is uniformly bounded from below, and all its partial derivatives up to order 3 are uniformly bounded from above. Thus, by Lemma~\ref{vdCorput}, 
\[
 \varphi_{\lambda_\tau+i\eta}(a_\delta) \ll (\tau\delta)^{-\frac12} = (\|\lambda_\tau\|\cdot \|\log a_\delta\|)^{-\frac12}.
\]
\hfill $\square$

\section{The Siegel upper half plane and $\Sp_4(\R)$}\label{sec2}  

In the following sections we set up the scene for the proof of Theorem \ref{mainthm}. From now on we write 
\[
 G \sceq \Sp_4(\R) = \left\{ M\in\GL_4(\R) \ \left\vert\ MJM^\top = J\right.\right\}
\]
with $J \sceq \textmat{}{I_2}{-I_2}{}$. A standard choice for an Iwasawa decomposition $G =KAN =NAK$ is 
\begin{align*}
 K & = \left\{ \left(\begin{matrix} B & C\\ -C & B \end{matrix}\right) \in O(4)\right\},
\\
A & = \left\{\diag\left(e^{t_1}, e^{t_2}, e^{-t_1}, e^{-t_2}\right) \left\vert\ t_1, t_2 \in \R \vphantom{\diag\left(e^{t_1}, e^{t_2}, e^{-t_1}, e^{-t_2}\right)} \right.\right\}
\intertext{and}
N & = \left\{ \left( \begin{matrix} B & C\\ & B^{-\top}\end{matrix}\right) \ \left\vert\ BC^{\top} = CB^{\top},\ \text{$B$ unit upper triangular}\vphantom{\mat{A}{B}{}{A^{-\top}}}\right.\right\}.
\end{align*}
We note that $K\cong U(2)$ via
\begin{equation}\label{compact}
\mat{B}{C}{-C}{B}\mapsto B + iC.
\end{equation}
The Lie algebra of $A$ is 
\[
\mathfrak{a} \sceq \{ \diag(t_1,t_2, -t_1,-t_2)\mid t_1,t_2\in\R\}.
\]
For $j=1, 2$, we define $e_j \in \mf a^*$ by 
\[
e_j(\diag(d_1, d_2, -d_1, -d_2) ) = d_j.
\]
The $8$ (restricted) roots of $(\mathfrak{ g},\mathfrak{ a})$ are $\Sigma \sceq  \{ \pm e_1 \pm e_2,  \pm 2e_1, \pm 2e_2\}$ (a $C_2$ root system). One easily checks that the multiplicity of each root is $1$. A standard choice of positive roots is  
\[
\Sigma^+ \sceq   \{  e_1 \pm e_2,    2e_1,  2e_2\}.
\]
From \eqref{GK} one  checks that
\begin{equation}\label{cexplicit}
|\boldc(\lambda ) |^{-2} =  \left(\frac{\pi}{4}\right)^2 \gamma(\lambda_1) \gamma(\lambda_2)\gamma(\lambda_1 + \lambda_2)\gamma(\lambda_1- \lambda_2) \ll \| \lambda\|^4, \quad \text{where $\gamma(x) \sceq x \tanh(\pi x/2)$.}
\end{equation}

 Let 
\[
\mc H \sceq \left\{Z = X + i Y \in {\rm Mat}_2(\C)\ \left\vert\  Z = Z^{\top} \text{ and } Y > 0\right.\right\}
\]
denote the Siegel upper half space with Riemannian metric determined by the line element $ds^2 = {\rm Tr}(dZ\, Y^{-1} \, d\bar{Z} \, Y^{-1})$. Here, $Y>0$ means that the matrix $Y$ is positive definite. The action of $G$ on $\mc H$ is given by
\[
\begin{pmatrix} A & B\\C & D \end{pmatrix}.Z= (AZ + B)(CZ + D)^{-1},
\]
which induces an isometry of manifolds and $G$-spaces between $G/K$ and $\mc H$ via $gK \mapsto g.iI_2$. One representative in $G$ of a point $Z  = X+ iY \in \mc H$  is given by 
\begin{equation}\label{point1} 
  g = \begin{pmatrix} I_2 & X\\ & I_2\end{pmatrix}\begin{pmatrix}V & \\ & V^{-1}\end{pmatrix} \in G
\end{equation}
where $V$ is the unique symmetric positive definite matrix satisfying $V^{\top}V = Y$.

Let $\mathcal{D}(\mc H)$ denote the algebra of differential operators on $\mc H$ which are invariant under the left action of $G$. This is a commutative algebra of rank 2 which contains the positive definite Laplace-Beltrami operator $\Delta$ on $\mc H$. The Harish-Chandra isomorphism \cite[Chapter II.5]{He1} establishes a bijection between $\mathcal{D}(\mc H)$ and the algebra of Weyl group invariant polynomials in $\mf{a}^{\ast}_{\C}.$ In particular, the image of $\Delta$ is 
\begin{equation}\label{laplace}
\langle \halfsum, \halfsum \rangle + \langle \lambda, \lambda \rangle = \frac{5 + \lambda_1^2 + \lambda_2^2}{12} \in \C[\lambda_1, \lambda_2]
\end{equation}
for $\lambda=\lambda_1e_1+\lambda_2e_2\in\mf a^*_\C$.

\section{Hecke operators}\label{sec3}

If $\mathcal{M}$ is a set of matrices in 
\[
\GSp^+_4(\Q)  = \left\{ M\in \GL^+_4(\Q) \ \left\vert\ \text{$MJM^\top  = rJ$ for some $r\in \Bbb{Q}^{\ast}$}\right.\right\}
\]
that is left- and right-invariant under $\Gamma = \Sp_4(\Z)$ and a finite union $\bigcup_j \Gamma M_j$ of  left cosets (equivalently, $\mathcal{M}$ is a finite union of double cosets), then $\mathcal{M}$ defines the associated Hecke operator on functions $F :   \Gamma \backslash \mc H\rightarrow \Bbb{C}$ by
\[
T_{\mathcal{M}} \colon F \mapsto \sum_{j} F \left( \frac{1}{(\det M_j)^{1/4}} M_j\,\,\cdot\  \right).
\]
This definition extends in an obvious way to the vector space of formal linear combinations of such sets $\mathcal{M}$, and we obtain the symplectic Hecke algebra $\mathscr H$, a commutative algebra of  operators that commute with the elements in $\mathcal{D}(\mc H)$ and that are Hermitian with respect to the inner product
\begin{equation}\label{inner}
\langle F_1, F_2\rangle \sceq \int_{\Gamma \backslash \mc H} F_1(Z) \overline{F_2(Z)} \frac{dX\, dY}{(\det Y)^3}.
\end{equation}
We refer to \cite[Section 4]{Fr} or \cite[Section 3]{AZ} for proofs of these facts, historical remarks, and an introduction to the theory of Hecke algebras.   The composition of two such operators corresponds to the multiplication of two double cosets $\Gamma A \Gamma = \bigcup_j \Gamma A_j$,  $\Gamma B \Gamma = \bigcup_k \Gamma B_k$:
$$T_{\Gamma A \Gamma } \circ T_{\Gamma B \Gamma} \colon F \mapsto  \sum_{j, k}F\left(\frac{1}{(\det A_jB_k)^{1/4}}A_jB_k \,\cdot\ \right). $$
It is easy to see that
\begin{equation}\label{product}
T_{\Gamma A \Gamma } \circ T_{\Gamma B \Gamma} = \sum_{D} c_{D} T_{\Gamma  D \Gamma}
\end{equation}
where $D$ runs through a system of representatives of double cosets contained in $\Gamma A \Gamma B \Gamma$ and $c_D$ is the number of pairs $(j, k)$ such that $\Gamma D = \Gamma A_jB_k$. However, in explicit situations this formula is combinatorially cumbersome.

For $m \in \N$  let 
\begin{equation}\label{sm}
\hog(m) \sceq \left\{M \in \GSp^+_4(\Z) \ \left\vert\ M^{\top}JM =m J\right.\right\}.
\end{equation}
Then the $m$-th Hecke operator is given by $T(m) \sceq T_{\hog(m)}.$ For $(m_1, m_2) = 1$ we have  
\begin{equation}\label{coprime}
   T(m_1) T(m_2) = T(m_1m_2).
 \end{equation}  
For $r\in\N_0$, $0 \leq a \leq b \leq r/2$ and any prime $p$ define
$$T^{(r )}_{a, b}(p) \sceq T_{\Gamma \text{diag}(p^a, p^b, p^{r-b}, p^{r-a}) \Gamma}.$$
Then $T(p^r)$  can be decomposed as a sum over Hecke operators on individual double cosets:
\begin{equation}\label{Heckedecomp}
T(p^r) = \sum_{0 \leq a \leq b \leq r/2} T^{(r )}_{a, b}(p ).
\end{equation}
Note that for all $a \in \N_0$ we have 
\begin{equation}\label{cancel}
T^{(2a)}_{a, a}(p )  = {\rm id}, \quad \text{hence} \quad T^{(r )}_{a, b}(p ) =  T^{( r - 2a)}_{0, b-a}(p ) T^{(2a)}_{a, a}(p ) = T^{( r - 2a)}_{0, b-a}(p ).
\end{equation}

For a function $F : \Gamma \backslash \mc H \rightarrow \Bbb{C}$ that is an eigenfunction of the Hecke algebra $\mathscr H$, we denote by $\lambda(m, F)$ and  $\lambda^{(r )}_{a, b}(p, F)$  the Hecke eigenvalue of $F$ with respect to $T( m)$ and  $T^{(r )}_{a, b}(p)$, respectively. In order to construct an efficient amplifier for the pre-trace formula we need to understand the Hecke relations. 
The formal generating series for Hecke operators is given by (combine \cite[Theorem 2]{Sh} with \eqref{cancel} and \eqref{square} below)
$$\sum_{r=0}^{\infty} T(p^r) X^r = \frac{1 - p^2 X^2}{1 - T(p ) X +  (T(p )^2 - T(p^2) - p^2  )X^2   - p^3  T(p ) X^3 + p^6  X^4 }.$$
Comparing coefficients for $r=4$, we conclude   that
$$T(p^4) = (p^2+ 2p^3) T(p )^2 - T(p )^4 + p^2 T(p^2) + T(p^2)T(p )^2 + T(p^2)^2 - p^6$$
so that $\lambda(p, F)$, $\lambda(p^2, F)$ and $\lambda(p^4, F)$ cannot be simultaneously small; quantitatively: 
\begin{equation}\label{ampli}
|\lambda(p, F )| + \frac{1}{p^{3/2}}|\lambda(p^2, F)|  + \frac{1}{p^{9/2}}|\lambda(p^4, F)| \gg p^{3/2}. 
\end{equation}
(It has been proved in \cite[Lemma 5.2]{SV} that an inequality of this type exists in great generality.)

In order to analyze the amplifier in Section \ref{secpre} below, we need an explicit decomposition of $T(p^r)^2$ for $r = 1, 2, 4$ into double cosets.  The Hecke relation 
\begin{equation}\label{square}
T(p )^2 = T^{(2)}_{0, 0}(p ) + (p+1)T^{(2)}_{0, 1}(p ) + (p^3 + p^2 + p + 1)T^{(2)}_{1, 1}(p )
\end{equation}
(see e.g.\@ \cite[p.\@ 219]{123}) is well-known. Unfortunately very little explicit is in the literature for higher powers, and the computations become indeed very involved. The only reference -- which received rather negative reviews in mathscinet and Zentralblatt due to a somewhat sub-optimal presentation, but nevertheless the involved computations based on \eqref{product} are very  useful -- we are aware of is \cite[p.\@ 120]{Ko}.  We quote the following for $r \geq 2$ (add the three columns of the table in the middle of the page and combine with \eqref{cancel}):
\begin{equation}\label{kodama}
\begin{split}
T(p^r) T( p^2) &=  T^{(r+2)}_{0, 0}( p ) + (p+1) T^{(r+2)}_{0, 1}( p ) + (p^2+p+1) \sum_{b=2}^{(r+2)/2}  T^{(r+2)}_{0, b}( p ) \\
& +  (p^3+p^2+p+1) T^{(r+2)}_{1, 1}( p ) + (p^4 + 2p^3 + p^2 + p + 1)  \sum_{b=1}^{r/2}  T^{(r )}_{0, b}( p ) T^{(2)}_{1, 1}(p ) \\
&+ (p^6 + p^5 + 2p^4 + 2p^3 + p^2 + p + 1) \sum_{a=1}^{r/2} \sum_{b = 0 }^{(r-2a)/2} T^{(r-2a)}_{0, b}( p ) T^{(2a+2)}_{a+1, a+1}(p ).   
\end{split}
\end{equation}
With $r=2$ we obtain an exact formula for $T(p^2)^2$. To decompose $T(p^4)^2$ into a linear combination of Hecke operators $T^{(8)}_{a, b}( p)$ with $0 \leq a \leq b \leq 4$, we content ourselves with fairly crude upper bounds for the coefficients. 

To this end we introduce the following notation: if $$T_1 =  \sum_{0 \leq a \leq b \leq r/2} \gamma_1(a, b) T^{( r)}_{a, b}( p), \quad T_2 =  \sum_{0 \leq a \leq b \leq r/2} \gamma_2(a, b) T^{( r)}_{a, b}( p)$$
are two Hecke operators, we write $T_1 \leq T_2$ if $\gamma_1(a, b) \leq \gamma_2(a, b)$ for all $a, b$ (note that the coefficients $\gamma_j(a, b)$ in such decompositions are unique). With this notation we obtain from  \eqref{Heckedecomp} with $2r+2$ in place of $r$ and \eqref{kodama} with $2r$ in place of $r$  that 
\begin{equation}\label{kodamanew}
\begin{split}
T(&p^{2r+2})  \leq T(p^{2r}) T(p^2)\\
&  \leq 3p^2 \sum_{b \leq r+1} T_{0, b}^{(2r+2)}(p ) + 6p^4 \sum_{b \leq r} T_{0, b}^{(2r)}(p ) T_{1, 1}^{(2)}( p) + 9p^6 \sum_{1 \leq a \leq r} \sum_{b \leq r-a} T_{0, b}^{(2r-2a)}(p ) T_{a+1, a+1}^{(2a+2)}(p )\\
&\leq 10 \sum_{ s \leq r+1} p^{2r + 4 - 2s} \sum_{b \leq s}  T_{0, b}^{(2s)}(p ) T^{(2r+2-2s)}_{r+1-s, r+1-s}(p ). 
\end{split}
\end{equation}
This remains trivially true for $r=0$. With $r=1$, we conclude $T(p^4) \leq T(p^2)^2$, and hence $T(p^4)^2 \leq T(p^4)T(p^2) T(p^2).$ Applying the upper bound in \eqref{kodamanew} twice (with $r=2$ and $r = s$) in connection with   \eqref{Heckedecomp}, we obtain
\begin{equation}\label{square4}
\begin{split}
T(p^4)^2& \leq T(p^4)T(p^2) T(p^2) \leq 10 \sum_{ s \leq 3} p^{8-2s} \sum_{b \leq s} T^{(2s)}_{0, b} (p )T^{(6-2s)}_{3 - s, 3-s}(p ) T(p^2)\\
& \leq 10 \sum_{ s \leq 3  }  p^{8-2s}  T(p^{2s} )T^{(6-2s)}_{3-s, 3-s}(p ) T(p^2)\\
& \leq 100 \sum_{ s \leq 3 }  \sum_{ \tau \leq s+1 }  p^{12-2\tau}   \sum_{b \leq \tau} T^{(2\tau)}_{0, b}(p )  T^{(8-2\tau)}_{4-\tau, 4-\tau}(p )  \\
 & \leq 400  \sum_{0 \leq \tau \leq 4 } p^{12-2\tau}   \sum_{b \leq \tau} T^{(2\tau)}_{0, b}(p ) T^{(8-2\tau)}_{4-\tau, 4-\tau}(p ) . 
\end{split}
\end{equation}
We rephrase \eqref{square}, the upper bound in \eqref{kodamanew} for $r=2$ and \eqref{square4} in terms of Hecke eigenvalues: for a Hecke eigenform $F $ and for $r \in \{1, 2, 4\}$ we have
\begin{equation}\label{squarefinal}
\lambda(p^r, F)^2 =   \sum_{0  \leq b \leq s \leq r}  c_{r, b, s}(p ) \lambda^{(2s )}_{0,  b}(p, F ), \quad c_{r, b, s}( p) \ll p^{3r - 2s}.
\end{equation}
The existence of such a decomposition is obvious (recall \eqref{cancel}), the important point is the bound on the coefficients which is best-possible for $b = s$ (for smaller $b$ better bounds are available, but we shall later need a uniform bound). 

It is tempting to perform all of these computations in the algebra of Weyl group invariant polynomials using the Satake isomorphism. Unfortunately, the computations are by no means easier, since it is very hard to compute explicitly the image of a given double coset. As it may be useful in other situations, we have collected some explicit formulae in the appendix. 

\section{Joint eigenfunctions, representations and spectral parameters}\label{sec:SMF}

A smooth function $F$ on $\Gamma\backslash\mc H$ that is a joint eigenfunction of the algebra $\mc D(\mc H)$ of $G$-invariant differential operators on $\mc H$ and of moderate growth is commonly called a (Siegel) Maa{\ss} wave form for $\Gamma$. It is not necessarily square integrable, e.g.\@ the (non-holomorphic) Eisenstein series constitute examples of not square integrable Maa{\ss} wave forms. However, if $F$ additionally satisfies the regularity property
\[
 \int_{(\Gamma\cap N_j)\backslash N_j} F(n.Z) d_jn = 0
\]
for all $Z\in \mc H$ and the unipotent radicals $N_j$ (with Haar measure $d_jn$) of the proper parabolic subgroups $P_j$ of $G$, then $F$ is called cuspidal or a Siegel Maa{\ss} cusp form and it is in particular  square integrable. \\

In order to apply the trace formula, we need a spectral decomposition of $L^2(\Gamma\backslash \mc H)$ respecting the Hecke algebra. This is best done adelically, and we refer to \cite{AS} for a corresponding dictionary. Let $\Bbb{A}$ be the adele ring of $\Q$. It follows from  Langlands' monumental theory of Eisenstein series (see in particular \cite{Langlands_funceq}), nicely summarized in \cite{Ar}, that the space $L^2({\rm GSp_4}(\Q)\backslash {\rm GSp_4}(\Bbb{A}))$ has a ${\rm GSp}(\Bbb{A})$-equivariant decomposition into a direct sum, parametrized by (classes of) parabolic subgroups and irreducible cuspidal automorphic representations of their Levi subgroups, of direct integrals. Each irreducible representation occurring in this decomposition is factorizable into local components, and for almost all places $v$, it contains a unique (up to scalars) $K_v$-fixed vector \cite{Fl}. In our situation, all representations are  unramified at all finite places and hence generated by a Hecke eigenform. 

We re-state in classical language that there exists a spectral decomposition 
\begin{equation*}\label{spec}
L^2(\Gamma\backslash \mc H) = L_{\pp}^2(\Gamma\backslash \mc H) \oplus  
L_{\ac}^2(\Gamma\backslash \mc H)
\end{equation*}
where $L_{\pp}^2(\Gamma\backslash\mc H)$ and $L_{\ac}^2(\Gamma\backslash\mc H)$ are the subspaces corresponding to the pure point spectrum and the absolutely continuous spectrum, respectively.  We write this decomposition   as
\begin{equation}\label{spec1}
L^2(\Gamma\backslash \mc H) = \int V_{\varpi} d\varpi
\end{equation}
where each $V_{\varpi}$ is a one-dimensional space generated by a (not necessarily square integrable) joint eigenfunction, i.e.\ an eigenfunction of the Hecke algebra $\mathscr H$ and of the algebra $\mc D(\mc H)$, 
in the sense that each function in the $L^2$-space on the left hand side decomposes into a convergent sum
and integral of functions from each subspace $V_{\varpi}$, and a corresponding Plancherel formula holds.

The space $L_{\pp}^2(\Gamma\backslash\mc H)$ is spanned by the square integrable Siegel Maa{\ss} wave forms for $\Gamma$ and decomposes further into 
\[
 L_{\pp}^2(\Gamma\backslash\mc H) = L_{\cuspidal}^2(\Gamma\backslash\mc H) \oplus L_{\resid}^2(\Gamma\backslash\mc H),
\]
where $L_\cuspidal^2(\Gamma\backslash\mc H)$ is spanned by the Siegel Maa{\ss} cusp forms, and $L_\resid^2(\Gamma\backslash\mc H)$ is spanned by square integrable iterated residues of the Eisenstein series (see  \cite{Kim} for the determination of the residual spectrum). \\

We emphasize that the Hecke eigenvalues $\lambda(p, F)$ and $\lambda(p^2, F)$, and hence all $\lambda_{a, b}^{( r)}(p, F)$, of a joint eigenfunction $F \in V_{\varpi}$ are \emph{real}. For $L^2$-functions this follows trivially from the self-adjointness of the Hecke operators with respect to the inner product \eqref{inner}. In general this is a consequence of the fact that the corresponding local representation $\varpi_p$ is unitary, unramified and with trivial central character. If $\alpha_p,  \beta_p$ denote the local Satake parameters, then $\lambda(p, F) = p^{3/2}(x+y)$ and  $\lambda(p^2, F) = p^3(x^2 + xy + y^2)$ with $x = \alpha_p + \alpha_p^{-1}$, $y = \beta_p + \beta_p^{-1}$, and an inspection of \cite[Table A.2]{RSch} or \cite[Proposition 3.1]{PS} shows that these numbers are real in all cases.\\

The joint eigenfunctions $F \in V_{\varpi}$ arise as vectors in an induced representation from the  parabolic subgroup $NA$ of the extension of a character  
\[
\chi \colon A \rightarrow \C^{\times}, \quad \diag(e^{t_1}, e^{t_2}, e^{-t_1}, e^{-t_2}) \mapsto e^{it_1\mu_1 + it_2\mu_2}.
\]
In this way we can identify the collection of functions in the various $V_{\varpi}$ with a  subset of $\mf a^*_{\C}/W$, where we associate to each joint eigenfunction $F$ the linear form (the spectral parameter)  $\mu =  \mu_1 e_1 +  \mu_2e_2 \in    \mf a^*_{\C} /W$ (or simply $\mf a^*_{\C}$) that contains the (archimedean) Langlands parameters. By \eqref{laplace}, the Laplace eigenvalue of $F$ is then given by 
\begin{equation}\label{laplaceeig}
\lambda_{F}   = \frac{5 + \mu_1^2 + \mu_2^2}{12}. 
\end{equation}

A Weyl law of the form 
\[
\dim \spann\{F\in L^2_\cuspidal(\Gamma\backslash\mc H) \mid \lambda_{ F}  \leq T\} \sim \text{const.}\cdot T^3
\]
is known \cite{LV}, rendering the statement of Theorem~\ref{mainthm} to be non-void. We remark that in our normalization (see also \eqref{defLambda} below) the tempered spectrum has real parameters $\mu \in \mf a^*$. By \cite{Nzoukoudi}, the spectral parameters (resp.\@ representatives of their Weyl group orbits) of all irreducible unitary representations occurring in \eqref{spec1} are contained in 
\begin{equation}\label{defLambda}
 \Lambda\sceq \mf a^* \cup \{ \mu=(\mu_1,\mu_2)\in \mf a^*_\C : \text{$\|\Ima \mu\|\leq \spectrbound$ and ($\mu_1 = -\overline{\mu_2}$  or $\mu_1\in \R \cup i\R$, $\mu_2\in   i\Bbb{R}$)}\}.     
\end{equation}
(Unlike in the case of the group ${\rm GL}_2$, the set of exceptional parameters $\mu$ for a given lattice is not known to be finite.)

\section{Pre-trace formula and amplification}\label{secpre}

The first step of Selberg's celebrated trace formula \cite{Se} is the spectral expansion of an automorphic kernel, resulting in the following pre-trace formula. We recall the spectral decomposition \eqref{spec1} and denote by $F_{\varpi}$ the normalized generator of $V_{\varpi}$ and   by   $\mu_{\varpi} \in \Lambda \subseteq \mathfrak{a}^{\ast}_{\Bbb{C}}$ its spectral parameter. For any test function $f \in C_c^{\infty}(G\dq K)$ we have 
\begin{equation}\label{pretrace}
\int  \wt{f}(\mu_{\varpi})F_{\varpi}(x)\overline{F_{\varpi}(y)} \,  d\varpi  = \sum_{\gamma \in \Gamma}f(x^{-1}\gamma y) \quad\text{for $x,y\in G$},
\end{equation}
where $\wt f$ is the spherical transform defined in \eqref{sphericaltransform}. The functions $F_{\varpi}$ are understood here as right-$K$-invariant functions on $\Gamma\backslash G$.  Let $F_{0} \in L^2(\Gamma \backslash \mathcal{H})$ be the $L^2$-normalized joint eigenfunction with spectral parameter $\mu_{0} \in \Lambda$ whose supremum norm we want to bound.  

\subsection{Choice of  test function} In order to prove Theorem \ref{mainthm}, we choose a specific test function $f \in C_c^\infty(G\dq K)$ to be used in \eqref{pretrace} adapted to the spectral parameter $\mu_0$. The choice of this test function is done indirectly by choosing its spherical transform $\mc H_W(\mf a^*_\C)$ with sufficiently small support about the identity $o\in G\dq K$ and applying an inverse transformation. We prove upper bounds on $f$ using Proposition~\ref{paley} in connection with Theorem \ref{sphericalbound}. Recall that without loss of generality  we may assume $\| \mu_{0} \|$ to be sufficiently large. The test function and its spherical transform will only depend on the real part 
\begin{equation}\label{defmu}
 \mu \sceq \Rea \mu_{0} \  \in \mf a^*
\end{equation}
of the spectral parameter of $F_{0}$.

Let $\psi \in \mc H^1_W(\mf a^*_\C)$ be a fixed Paley--Wiener function  such that 
\begin{enumerate}[(a)]
\item $\psi$ is nonnegative on $\mf a^*$, so in particular $\psi(\overline{\lambda}) = \overline{\psi(\lambda)}$, 
\item $\Rea \psi \geq 1$ in a ball of radius $\spectrbound$ about $0 \in \mf a^*_\C$.
\end{enumerate}
(By choosing $\psi\in \mc H_W^{C_0/2}(\mf a^*_\C)$ where $C_0$ is the global constant appearing in the proof of Theorem~\ref{sphericalbound}, we could have made the argument independent of \cite{DKV}.) For  $\lambda\in\mf a^*_\C$ we define
\begin{equation}\label{defftilde}
 \wt f_\mu(\lambda) \sceq \Big( \sum_{w\in W} \psi(\mu-w.\lambda) \Big)^2.
\end{equation}
Clearly, $\wt f_\mu \in \mc H_W^{2}(\mf a^*_\C)$. Further, it is easy to see that  
\begin{equation}\label{pos}
  \wt{f}_{\mu}(\lambda) \geq 0 \text{ for all } \lambda \in \Lambda 
\end{equation} 
 and 
\begin{equation}\label{big}
\wt{f}_{\mu}(\lambda) \geq 1 \text{ for all } \lambda \in \Lambda \text{ whenever } \Rea \lambda = \mu. 
\end{equation}
The latter is obvious for $\lambda \in \mf a^*$. We recall the classification of possible exceptional parameters in \eqref{defLambda}. For an exceptional spectral parameter of the form $\lambda = (x+ iy) e_1 + (-x+iy) e_2$  and  $\mu = xe_1-xe_2$ with $x \in \Bbb{R}$ sufficiently large and $y \in [-\sqrt{5/2}, \sqrt{5/2}]$  we have 
\begin{align}
\nonumber \wt f_\mu(\lambda) &= 4 \Big( \Rea \psi\big( i(y,y) \big) +  \Rea \psi\big( (2x,0) + i(-y,y) \big) 
\\
\nonumber & \qquad 
+ \Rea \psi\big( (0,-2x) + i(-y,y) \big) + \Rea \psi\big( (2x,-2x) + i(y,y) \big)\Big)^2
\\
\nonumber & \geq 3 \left| \Rea \psi\big( i(y,y) \big)\right|^2 \geq 1 
\end{align}
 by the rapid decay of $\psi$ for large $x$ and   $ \psi(\overline\lambda) = \overline{\psi(\lambda)}$. For an exceptional spectral parameter of the form $\lambda = iye_1 + x e_2$ and $\mu=xe_2$   we have by the same argument
\begin{align}
\nonumber \wt f_\mu(\lambda) &= 4 \Big( \Rea \psi\big( i(0,y) \big) + \Rea \psi\big( (x,x)+ i(y,0)\big) 
\\
\nonumber & \qquad + \Rea\psi\big( (2x,0) + i(0,y)\big) + \Rea \psi\big( (x,-x) + i(y,0)\big)\Big)^2
\\
\nonumber & \geq 3 \left| \Rea \psi\big( i(0,y)\big)\right|^2\geq 1.
\end{align}
Exceptional spectral parameters with $\lambda_1, \lambda_2$ both purely imaginary are bounded and hence do not meet the condition $\Rea \lambda = \mu$. 

Finally, $\wt f_\mu\in\mc H_W(\mf a^*_\C)$ immediately implies
\begin{equation}\label{decay}
   \wt{f}_{\mu}(\lambda) \ll_A \max_{w \in W}(1+ \|\mu - w.\lambda\|)^{-A}
\end{equation}
for $\lambda\in\mf a^*$.

By Proposition~\ref{paley} the inverse spherical transform   $  f_\mu$  of $ \wt f_\mu$ has support in $\overline{B^A_{2}(o)}$, independent of $\mu$. Combining Theorem~\ref{sphericalbound}, Proposition~\ref{paley}, \eqref{cexplicit} and \eqref{decay} and recalling the notation \eqref{defC} for the Cartan projection, we conclude
\begin{equation}\label{mainbound}
  f_{\mu}(g) \ll \| \mu \|^{4} (1 + \| \mu \| \cdot \|C(g) \|)^{-1/2}.  
\end{equation}

\subsection{Construction of the amplifier} With the choice \eqref{defftilde} we return to \eqref{pretrace} and construct a suitable amplifier. Given a double coset $\Gamma M \Gamma = \bigcup_i \Gamma M_i$ with $M, M_i \in \GSp^+_4(\Z)$, we apply \eqref{pretrace} for the elements $(x, y) = (g, (\det M_i)^{-1/4} M_i g) \in G\times G$  obtaining
\[
 \int \lambda(M, F_{\varpi}) \wt{f}_{\mu}(\mu_{\varpi})|F_{\varpi}(g)|^2 \, d \varpi    = \sum_{\gamma \in \Gamma M \Gamma}f_{\mu}(g^{-1} \wt{\gamma} g)
\]
where $\wt{\gamma} \sceq (\det \gamma)^{-1/4}\gamma$ and 
 $\lambda(M, F_{\varpi})$ is the eigenvalue of $F_{\varpi}$ with respect to $T_{\Gamma M \Gamma}$. 

Let $L \geq 5$ be a parameter and let $\mathcal{P}$ be the set of primes in $(L, 2L]$. Define $x(n) \sceq \text{sgn}(\lambda(n, F_{0}))$ for $n\in\N$ and let 
\begin{equation*} 
A_{\varpi} \sceq \Bigl(\sum_{l \in \mathcal{P}}x(l) \lambda(l, F_{\varpi})\Bigr)^2 + \Bigl(\sum_{l \in \mathcal{P}} x(l^2) l^{-3/2} \lambda(l^2, F_{\varpi})\Bigr)^2 + \Bigl(\sum_{l \in \mathcal{P}} x(l^4) l^{-9/2} \lambda(l^4, F_{\varpi})\Bigr)^2 \geq 0.
\end{equation*}
By \eqref{ampli} and the prime number theorem we have
\begin{equation}\label{ampli2}
\begin{split}
A_{0}& = \Bigl(\sum_{l \in \mathcal{P}}  | \lambda(l, F_{0})|\Bigr)^2 + \Bigl(\sum_{l \in \mathcal{P}}  l^{-3/2} |\lambda(l^2, F_{0})|\Bigr)^2 + \Bigl(\sum_{l \in \mathcal{P}}   l^{-9/2} |\lambda(l^4, F_{0})|\Bigr)^2 \\
&\geq \frac{1}{3}  \Bigl(\sum_{l \in \mathcal{P}}  | \lambda(l, F_{0})|+  l^{-3/2} |\lambda(l^2, F_{0})| +   l^{-9/2} |\lambda(l^4, F_{0})|\Bigr)^2  \gg L^5 (\log L)^{-2}. 
\end{split}
\end{equation}
Combining \eqref{coprime} and \eqref{squarefinal}, we can expand $A_\varpi$:
\begin{equation}\label{ampli3}
\begin{split}
A_{\varpi} & = \sum_{l_1 \not= l_2} \left(x(l_1l_2) \lambda(l_1l_2, F_{\varpi})  + \frac{x(l_1^2l_2^2)}{(l_1l_2)^{3/2}} \lambda(l_1^2l_2^2, F_{\varpi}) + \frac{x(l_1^4l_2^4)}{(l_1l_2)^{9/2}} \lambda(l_1^4l_2^4, F_{\varpi})\right)\\
& + \sum_{0 \leq  b \leq s \leq 4} \sum_{l} \xi_{b, s}(l) \lambda_{0, b}^{( 2s)}(l, F_{\varpi} )  \end{split}
\end{equation}
where $\xi_{b, s}(l) \ll l^{3 - 2s}$. By \eqref{pos}, \eqref{big}, \eqref{ampli2} and \eqref{ampli3} we conclude that 
\begin{equation}\label{pretrace1}
\begin{split}
 \frac{L^{5 }}{(\log L)^2}  |F_{0}(g)|^2 &\ll  \int A_{\varpi}  \wt{f}_{\mu}(\mu_{\varpi}) |F_{\varpi}(g)|^2 \, d\varpi   \\
 & \ll \sum_{r \in \{1, 2, 4\}} \sum_{l_1 \not= l_2 \in \mathcal{P}(L)} (l_1l_2)^{\frac{3}{2}(1 - r)} \sum_{\gamma \in \hog({l_1^rl_2^r})}   |f_{\mu}(g^{-1}\wt{\gamma} g)|  \\
 & \quad\quad\quad +\sum_{0 \leq r \leq 4} \sum_{l \in \mathcal{P}(L)}  l^{3 - 2r}  \sum_{\gamma \in \hog(l^{2r})}  |f_{\mu}(g^{-1}\wt{\gamma} g)|  
  \end{split}
\end{equation}
with $S(m)$ as in \eqref{sm}. 

\subsection{Proof of Theorem \ref{mainthm}} In order to prove Theorem \ref{mainthm},  we  estimate the right hand side of \eqref{pretrace1}. To this end, we fix some small $0 < \delta_0 < 1$. We estimate the number of matrices $\gamma$ with $\| C(g^{-1}\wt \gamma  g  )\| \geq \delta_0$ trivially and obtain a saving from the decay of $f_{\mu}$ as specified in \eqref{mainbound}. On the other hand, we will prove a strong bound for the number of matrices $\gamma$ with $\| C(g^{-1}\wt \gamma  g )\| \leq \delta_0$, and for those we estimate $f_{\mu}$ trivially.  With this in mind, we define
\begin{equation}\label{defS}
\mathscr{S}(g)_{\delta}[m] \sceq \left\{\gamma \in\hog(m) \ \left\vert\  \|\wt\gamma   - gKg^{-1} \| \leq  \delta \vphantom{\hog(m)}  \right.\right\}
\end{equation}
for $m \in \N$ and  $\delta>0$. Since $g$ varies in a fixed compact set $\cptset$, we make the observation that $\| C(g^{-1}\wt \gamma  g) \| \leq \delta$ for some $\delta \ll 1$ and  $\det \gamma = m^2$ imply $\wt {\gamma} \in g K g^{-1} + O_{\Omega}(\delta)$ and hence $\gamma \in \mathscr{S}(g)_{c\delta}[m]$ for some constant $c>0$ depending on $\cptset$.  

The next section is devoted to a bound for the cardinality of $\mathscr{S}(g)_{\delta}[m]$ uniformly in $\delta$ and $m$. Our principal result in this direction is 
\begin{prop}\label{prop1} There exist $\eta, B>0$ such that for all $\delta \ll 1$, $\eps>0$, $m\in\N$ and $g\in G$ we have
\[
 \#\mathscr{S}(g)_\delta[m]\ll_{g,\eps} m^{1+\varepsilon} \left(1 + \delta^{\eta}m^B\right).
\]
The implied constant does not depend on $m$ and $\delta$.
\end{prop}

Taking this for granted (a proof follows in Section~\ref{sec7} below), it is now a simple matter to prove Theorem \ref{mainthm}. 

\begin{proof}[Proof of Theorem~\ref{mainthm}]
We insert \eqref{mainbound} and Proposition~\ref{prop1} into \eqref{pretrace1}. The constant $B$ used below is global and possibly larger than the one in Proposition~\ref{prop1}, the variables $\delta, L $ will be fixed below. As mentioned before, we can assume that the spectral parameter  $\mu_{0} $ (and hence its real part $\mu$) of our form $F_{0}$ is sufficiently large.  

The contribution of the matrices $\gamma$ with $\| C(g^{-1}\wt \gamma  g )\| \geq \delta$ to the right hand side of \eqref{pretrace1} is 
$$\ll_{\cptset}  \| \mu \|^{7/2} \delta^{-1/2} L^{B},$$
and the contribution of the matrices $\gamma$ with $\| C(g^{-1}\wt \gamma  g)\| \leq \delta$ to the right hand side of \eqref{pretrace1} is 
$$\ll_{\eps, \cptset} \| \mu \|^4 L^{4+\varepsilon} (1 + \delta^{\eta}L^{B}).$$
Upon choosing  
$$\delta^{1/2 + \eta}  \sceq  \| \mu \|^{-1/2} \quad\text{and}\quad  L \sceq \left\lceil\delta^{-\frac{\eta}{B}}\right\rceil$$
we obtain from \eqref{pretrace1} that
\begin{displaymath}
   |F_{0}(g)|^2 \ll_{\cptset,\eps} \| \mu \|^{4} L^{-1+\eps} (\log L)^2 \ll \|\mu\|^4 L^{-\frac34}. 
\end{displaymath}
Thus,
\begin{displaymath}
  F_{0}(g) \ll_{\eps,\cptset} \| \mu \|^{2 - \frac{3\eta}{4B(1 + 2\eta)} + \varepsilon}. 
\end{displaymath}
Theorem \ref{mainthm} now follows from \eqref{laplaceeig}. 
\end{proof}

\section{Diophantine Analysis}\label{sec7}

It remains to prove Proposition \ref{prop1} which is the purpose of this section. Let
$$k = \left(\begin{matrix} B & C\\ -C & B \end{matrix}\right) \in \sqrt{m} K.$$
This matrix is symplectic and orthogonal, i.e.\@ 
\begin{equation}\label{properties}
k^\top k = m I_4 \quad \text{and} \quad k^\top Jk = mJ.
\end{equation} 
Let 
\begin{equation*} 
Q \sceq (g g^{\top})^{-1},
\end{equation*}
which is a positive definite matrix depending only on $g\in G$. To simplify notation we set 
\[
 Q = (q_{ij}) = (q_1\ q_2\ q_3\ q_4),
\]
where $q_i=(q_{1i}, q_{2i}, q_{3i}, q_{4i})^\top$ are column vectors. Further we set 
\[
 Q_1 \sceq (q_1\ 0\ q_3\ 0) \quad\text{and}\quad Q_2\sceq (0\ q_2\ 0\ q_4),
\]
hence $Q=Q_1+Q_2$.

\begin{prop}\label{propstab1} Let $\delta \ll 1$. For any $g\in G$, the set $\mathscr{S}(g)_{\delta}[m] $ defined in \eqref{defS} consists of matrices $$ \gamma = \left(\begin{matrix} r_1 & s_1 & * & * \\   r_2 & s_2 & * & * \\ r_3& s_3 &* &* \\ r_4 & s_4 & * & * \end{matrix}\right)$$
with the following properties:
\begin{enumerate}[{\rm (a)}]
\item\label{propstab1a} All entries of the matrix are $\ll_g m^{1/2} $. 
\item\label{propstab1b} The vectors $ {r} = (r_1, r_2, r_3, r_4)^\top$ and $ {s} = (s_1, s_2, s_3, s_4)^\top$ determine each of the remaining entries of the matrix up to $O_g(\delta m^{1/2})$. 
\item\label{propstab1c} We have
\begin{equation}\label{mat1a}
 {r}^{\top} Q {r} = m q_{11} + O_g(\delta m ),  \quad  {s}^{\top} Q {s} = m q_{22} + O_g(\delta m ).
\end{equation}
\item\label{propstab1d} If  $ {r}^{\top}Q_1 {r} \asymp m$, then 
\begin{equation}\label{mat4a}
\begin{split}
s_1 & = \frac{s_2 \cdot  {r}^{\top}A_{11} {r} + s_4 \cdot  {r}^{\top}A_{12} {r}   + m r_1 q_{12}}{ {r}^{\top}Q_1 {r}} + O_g(\delta m^{1/2}),\\
   s_3 &  = \frac{ s_2 \cdot  {r}^{\top}A_{21} {r} + s_4 \cdot  {r}^{\top}A_{22} {r}   + m r_3 q_{12}}{ {r}^{\top}Q_1 {r}} + O_g(\delta m^{1/2})
   \end{split}
\end{equation}
with $A_{11} = (-q_2\ 0\ 0\ -q_3)$, $A_{12}=(-q_4\ q_3\ 0\ 0)$, $A_{21}=(0\ 0\ -q_2\ q_1)$ and $A_{22} = (0\ -q_1\ -q_4\ 0)$.
If $ {r}^{\top}Q_2 {r} \asymp m$, then
\begin{equation}\label{mat5a}
\begin{split}
s_2 & = \frac{s_1 \cdot  {r}^{\top}B_{11} {r} + s_3 \cdot  {r}^{\top}B_{12} {r}   +m r_2 q_{12}}{ {r}^{\top}Q_2 {r}} + O_g(\delta m^{1/2}),\\
s_4 &  = \frac{ s_1 \cdot  {r}^{\top}B_{21} {r} + s_3 \cdot  {r}^{\top}B_{22} {r}   + m r_4 q_{12}}{ {r}^{\top}Q_2 {r}} + O_g(\delta m^{1/2})
\end{split}
\end{equation}
with $B_{11} = (0\ -q_1\ -q_4\ 0)$, $B_{12} = (q_4\ -q_3\ 0\ 0)$, $B_{21} = (0\ 0\ q_2\ -q_1)$ and $B_{22}=(-q_2\ 0\ 0\ -q_3)$.
\end{enumerate}
\end{prop}

\begin{proof} Part \eqref{propstab1a} is obvious from the definition \eqref{defS}.\\
Part \eqref{propstab1b}  follows from
$$g k g^{-1}   = g \left(\begin{matrix} B & C\\ - C & B\end{matrix}\right) g^{-1}  = \left(\begin{matrix}  VBV^{-1} - X V^{-1} C V^{-1} & * \quad \\ -V^{-1}CV^{-1} & * \quad  \end{matrix} \right)$$
where we used the notation \eqref{point1}. 
Hence for fixed $X$, $V$, the first two columns of a matrix  $\gamma \in m^{1/2}gKg^{-1} + O(\delta m^{1/2})$ determine $B$ and $C$  up to $O(\delta m^{1/2})$, and hence the remaining two columns up to $O(\delta m^{1/2})$. \\
By \eqref{properties} and part \eqref{propstab1a} we have   $\gamma^{\top} Q \gamma = mQ + O_g(\delta m)$ for $\gamma \in \mathscr{S}(g)_{\delta}[m]$   which implies part \eqref{propstab1c}.\\
The same argument shows $\gamma^{\top} J \gamma = mJ + O_g( \delta m)$ for $\gamma  \in \mathscr{S}(g)_{\delta}[m]$, hence 
\begin{equation}\label{matalt}
   {r}^{\top} J  {s} = O_g(\delta m) \quad \text{and}\quad  {r}^{\top} Q  {s} = m q_{12} + O_g( \delta m).
\end{equation}  
If  $ {r}^{\top}Q_1 {r} \asymp m$ (in particular $\not= 0$), we solve the linear system  \eqref{matalt} for $s_1$ and $s_3$, obtaining \eqref{mat4a}. If $ {r}^{\top}Q_2 {r} \asymp m$, the same argument gives \eqref{mat5a}. This completes the proof of part \eqref{propstab1d}.
\end{proof}
  
We see that for sufficiently small $\delta$, the matrices in question are essentially characterized by the 8 variables $r_j, s_j$ that satisfy the four quadratic equations \eqref{mat1a} and \eqref{matalt}. Hence geometrically we need to study the lattice points in a $\delta$-neighbourhood of the intersection of four quadrics in 8 variables. Roughly speaking, we will choose $r_1, r_2$ freely, then $r_3, r_4$ are essentially determined by the first equality in \eqref{mat1a} and the theory of binary quadratic forms. Once $ {r}$ is fixed, we substitute \eqref{mat4a} or \eqref{mat5a} into the second equation of \eqref{mat1a}, obtaining again a binary problem that essentially fixes $ {s}$. Hence  the cardinality of the integral matrices in $\mathscr{S}(g)_{\delta}[m]$ is  $O(m^{1+\varepsilon})$ for sufficiently small $\delta$.  At least for very small $\delta$ and odd positive  integers $m$, the bound of Proposition \ref{prop1} is essentially best possible, since $\mathscr{S}({\rm id})_{\delta}[m]$ contains, for 
every $\delta > 0$, all the $\asymp m$ matrices of the form 
\begin{displaymath}
 \left(\begin{matrix} a_1 & a_3 & a_2  & a_4  \\
-a_3  & a_1  & a_4  & -a_2   \\
-a_2 & -a_4   & a_1   & a_3 \\
-a_4 & a_2  &  -a_3  & a_1  
\end{matrix}\right), \quad a_1^2 + a_2^2 + a_3^2 + a_4^2   = m.
\end{displaymath}

In order to make these arguments rigorous in the following we  start with a multi-dimensional version of Dirichlet's approximation theorem \cite[Theorem 200]{HW}.
\begin{lemma}\label{lem1}
Let $\xi_1, \ldots, \xi_n$ be real numbers, $T > 1$. Then there exist integers $p_1, \ldots, p_n$ and a  positive integer $q \leq T$ such that $  |\xi_j - p_j/q| \leq (qT^{1/n})^{-1}$ for all $1 \leq j \leq n$. 
\end{lemma}

For a polynomial $P$ we denote by $H(P)$ the largest coefficient in absolute value.

The next lemma is standard.
\begin{lemma}\label{lem2}
\begin{enumerate}[{\rm (a)}]
\item\label{lem2a} Let $P(x, y) \in \Z[x, y]$ be a quadratic polynomial whose quadratic homogeneous part is positive definite.   Then the number of integral solutions to $P(x, y) = 0$ is $O_\eps( H(P)^{\varepsilon})$ for all $\varepsilon> 0$.\\
\item\label{lem2b} There exists a constant $C>0$ such that for each $\delta,D>0$ and each quadratic polynomial $P(x, y) \in \R[x, y]$ whose quadratic homogeneous part is positive definite with discriminant $|\Delta| \geq D$, the bound $|P(x, y)| \leq \delta$ implies $\max(|x|, |y|) \ll_D (\delta + 1 + H(P))^C$. 
\end{enumerate}
\end{lemma}

\begin{proof} We write $P(x, y) = ax^2 + bxy + cy^2 + d x + e y + f$ and $\Delta = b^2 - 4ac < 0$ (so in particular $a \not= 0$). We write $\xi = (be - 2cd)/\Delta$ and $\eta = (bd - 2ae)/\Delta$. One checks that
\begin{equation}\label{polyessential}
P(x, y) = \frac{(2a(x + \xi) + b(y + \eta))^2 - \Delta(y + \eta)^2}{4a} + P(-\xi, - \eta).
\end{equation}
Hence $P(x, y) = 0$ implies $X^2 - \Delta Y^2 = -4a \Delta^2 P(-  \xi, -   \eta)$ for certain integers $X, Y$. The number of solutions in $X, Y$ is at most the number of ideals in $\Bbb{Q}(\sqrt{\Delta})$ of norm $-4a\Delta^2P(-\xi, -\eta)$ which is bounded by $6$ times the number of divisors of $|4a\Delta^2P(-\xi, -\eta)|$. The well-known growth bound $d(n)= O_\eps(n^\eps)$  for the divisor function now implies part \eqref{lem2a}.  
 Part \eqref{lem2b} follows from straightforward estimates: first 
\[
 \delta \geq |P(x,y)| \geq \frac{|\Delta| |y+\eta|^2}{4|a|} - |P(-\xi,-\eta)|
\]
implies $|y| \ll_D (\delta + 1 + H(P))^C$. Using this bound in \eqref{polyessential} yields then the claimed bound for $|x|$. \end{proof}

\begin{cor}\label{cor3}
There exists a constant $A>0$ such that for each $\eps, \delta, D>0$ and each quadratic polynomial $P(x,y)\in\R[x,y]$ whose quadratic homogeneous part is positive definite with discriminant $|\Delta|\geq D$  we have
\[
 \#\{ (x,y)\in\Z^2 \mid |P(x,y)|< \delta \} \ll_{D,\eps} Z^\eps + \delta^{1/7} Z^A
\] 
where $Z = \delta + 1 + H(P )$. 
\end{cor}

\begin{proof} Let $(x,y)\in\Z^2$ with $|P(x,y)|<\delta$.   By Lemma \ref{lem2}\eqref{lem2b} we can assume that $\max(|x|, |y|) \ll_D Z^C$. Let $T>1$ be a parameter to be chosen later.  We approximate the six coefficients of $P$ by rational numbers with common denominator $q \leq T$. Using Lemma \ref{lem1}  and multiplying by $q$, we obtain 
\[
|\wt{P}(x, y)| \leq c_D(\delta T + Z^{2C} T^{-1/6}) \seqc R
\]
for some integral polynomial $\wt{P} \in \Z[x, y]$ of height $H(\wt{P}) \ll_D T H(P )$ and a constant $c_D>0$ only depending on $D$. For each integer $r \leq R $ we bound the number of solutions to $\wt{P}(x, y) - r = 0$ by Lemma \ref{lem2}\eqref{lem2a} getting a total of $\ll_{\eps,D}(1+R)  (R+TH( P))^{\varepsilon}$ solutions at most. 
We choose $$T  =  1 + \min(Z^{12C/7}  \delta^{-6/7}, Z^{12C}),$$
so that $1+R \ll_D 1 + \delta + \delta^{1/7}Z^{12C/7}$ and $R+TH( P) = Z^{O(1)}$.  This completes the proof. \end{proof}

We are now prepared for the 
\begin{proof}[Proof of Proposition \ref{prop1}]
We choose $r_1, r_2 \ll_g m^{1/2}$, and substitute the values into the first equation of \eqref{mat1a}. This provides us with the quadratic polynomial
\begin{align*}
 P(r_3,r_4) = q_{44}r_4^2 &+ 2q_{34}r_3r_4 + q_{33}r_3^2 + 2(q_{13}r_1+q_{23}r_2)r_3 + 2(q_{14}r_1 + q_{24}r_2)r_4 
 \\
&+ q_{11}r_1^2 + 2q_{12}r_1r_2 + q_{22}r_2^2 - mq_{11},
\end{align*}
whose quadratic homogeneous part is positive definite and $H(P)\ll_g m$. By \eqref{mat1a}, $|P(r_3,r_4)|\ll_g \delta m$. Hence Corollary \ref{cor3} shows that we have $\ll_{g,\eps} m^{\varepsilon} +(\delta m)^{1/7} m^{A} $ choices for $r_3, r_4$.  Without loss of generality let us assume that $ {r}^{\top} Q_1  {r} \asymp m$. We substitute \eqref{mat4a} into the second equation of \eqref{mat1a} and get a binary quadratic form $\wt P_r(s_2,s_4)$ whose coefficients, and also its discriminant, depend on $r$ (and $g$). Since $Q$ is positive definite, so is $\wt P_r$. Moreover, its shortest vector is trivially bounded below by the shortest vector of $Q$ which is bounded below by a constant depending only on $g$.  Minkowski's lower bound for the  discriminant of a quadratic form by its successive minima \cite[Chapter~12, Theorem~2.2]{Cassels} now shows  that the discriminant of the quadratic homogeneous part of $\wt P_r$ is bounded away from $0$ uniformly in $r$. Clearly, $H(\wt P_r) \ll_g m$. 
Then Corollary \ref{cor3} 
restricts the number of choices for $s_2, s_4$ to $\ll_{g,\eps} m^{\varepsilon} +(\delta m)^{1/7} m^{A}$. Now $s_1, s_3$ and the remaining $8$ entries are determined up to $O_g(\delta m^{1/2})$. This gives a total count for $\#\mathscr{S}(g)_\delta[m]$ of $\ll_{g,\eps} m^{1+\varepsilon} \left(1 + (\delta m )^{2/7}m^{2A}\right)   (1+\delta m^{1/2})^{10}$.
\end{proof}

\section{Appendix: Images of Hecke operators under the Satake map}

Given a double coset $\Gamma \text{diag}(p^a, p^b, p^{r-b}, p^{r-a}) \Gamma$, there exists a decomposition into left cosets
\[
\Gamma M \Gamma = \bigcup_j \Gamma M_j, \quad M_j = \left(\begin{matrix}A_j & \ast \\  & p^r A_j^{-\top} \end{matrix}\right), \quad A_j =  \left(\begin{matrix}p^{\alpha} & \ast\\  & p^{\beta} \end{matrix}\right).
\]
The image of $T_{\Gamma M \Gamma}$ under the Satake map is the polynomial
$$x_0^r \sum_{j} \left(\frac{x_1}{p}\right)^{\alpha}  \left(\frac{x_2}{p^2}\right)^{\beta}\in \C[x_0, x_1, x_2].$$  
This map is an algebra isomorphism between the $p$-part of the integral Hecke algebra and polynomials that are symmetric in $x_1,  x_2$ and invariant under the automorphisms
$$(x_0, x_1, x_2) \mapsto (x_0 x_1, 1/x_1, x_2) \quad \text{and} \quad (x_0, x_1, x_2) \mapsto (x_0 x_2, x_1, 1/x_2).$$

Table~\ref{extension} is compiled using the results of \cite[p.\ 120]{Ko} and comparing coefficients. Following \cite{RSh}, it is most efficient to use symmetrized polynomials. For $\textbf{a} = (a_1, a_2)$ with $0 \leq a_1 \leq a_2 \leq r/2$ we define the Weyl orbit as
\begin{displaymath}
\begin{split}
W_r(a_1, a_2) \sceq \{& (a_1, a_2), (a_2, a_1), (r - a_1, a_2), (a_2, r-a_1),\\
&  (a_1, r-a_2), (r-a_2, a_1), (r - a_1, r- a_2), (r-a_2, r - a_1)\}
\end{split}
\end{displaymath}
and 
$$\textbf{x}_r^{\textbf{a}} = \sum_{(b_1, b_2) \in W_r(a_1, a_2)} x_1^{b_1}x_2^{b_2}.$$

The entries for the Hecke operators of order $3, 4, 5, 6$ do not seem to be in the literature in explicit form and extend the matrix in \cite[p.\@ 238]{RSh}. 

 \newpage

\renewcommand{\arraystretch}{1.5}
\begin{table}[ht]
\begin{tabular}{l|l}
$T^{(1)}_{0, 0}$ & $x_0\textbf{x}_1^{(0, 0)}$ \\ \hline
$T^{(2)}_{0, 0}$ & $x_0^2\left(\textbf{x}_2^{(0, 0)} + \frac{p-1}{p} \textbf{x}_2^{(0, 1)} + \frac{2(p-1)}{p} \textbf{x}_2^{(1, 1)}\right)$\\ 
$T^{(2)}_{0, 1}$ & $x_0^2\left(\frac{1}{p}\textbf{x}_2^{(0, 1)}+ \frac{p^2-1}{p^3}\textbf{x}_2^{(1, 1)}\right)$\\  \hline
$T^{(3)}_{0, 0}$ & $x_0^3\left(\textbf{x}_3^{(0, 0)} + \frac{p-1}{p} \textbf{x}_3^{(0, 1)} + \frac{(p-1)(2p-1)}{p^2} \textbf{x}_3^{(1, 1)}\right)$\\ 
$T^{(3)}_{0, 1}$ & $x_0^3\left(  \frac{1}{p} \textbf{x}_3^{(0, 1)} + \frac{(p-1)(2p+1)}{p^3} \textbf{x}_3^{(1, 1)}\right)$\\ \hline
$T^{(4)}_{0, 0}$ & $x_0^4\left(\textbf{x}_4^{(0, 0)} + \frac{p-1}{p} \textbf{x}_4^{(0, 1)}+ \frac{p-1}{p} \textbf{x}_4^{(0, 2)} + \frac{(p-1)(2p-1)}{p^2} \textbf{x}_4^{(1, 1)} + \frac{2(p-1)^2}{p^2} \textbf{x}_4^{(1, 2)}  + \frac{(p-1)(3p^2 - 2p + 1)}{p^3} \textbf{x}_4^{(2, 2)} \right)$\\ 
$T^{(4)}_{0, 1}$ & $x_0^4\left(  \frac{ 1}{p}   \textbf{x}_4^{(0, 1)} +  \frac{ p-1}{p}   \textbf{x}_4^{(0, 2)} + \frac{2(p-1)}{p^2} \textbf{x}_4^{(1, 1)} + \frac{3(p-1)}{p^2} \textbf{x}_4^{(1, 2)}  + \frac{(p-1)^2(3p + 1)}{p^4} \textbf{x}_4^{(2, 2)} \right)$ \\
$T^{(4)}_{0, 2}$ & $x_0^4\left(  \frac{ 1}{p^2}    \textbf{x}_4^{(0, 2)} + \frac{p-1}{p^3} \textbf{x}_4^{(1, 1)} + \frac{p-1}{p^3} \textbf{x}_4^{(1, 2)}  + \frac{2(p-1) }{p^3} \textbf{x}_4^{(2, 2)} \right)$ \\ \hline
$T^{(5)}_{0, 0}$ & $x_0^5\left(\textbf{x}_5^{(0, 0)} + \frac{p-1}{p} \textbf{x}_5^{(0, 1)}+ \frac{p-1}{p} \textbf{x}_5^{(0, 2)} + \frac{(p-1)(2p-1)}{p^2} \textbf{x}_5^{(1, 1)} + \frac{2(p-1)^2}{p^2} \textbf{x}_5^{(1, 2)}  + \frac{(p-1)(3p^2 - 3p + 1)}{p^3} \textbf{x}_5^{(2, 2)} \right)$\\ 
$T^{(5)}_{0, 1}$ & $x_0^5\left(  \frac{ 1}{p}   \textbf{x}_5^{(0, 1)} +  \frac{ p-1}{p}   \textbf{x}_5^{(0, 2)} + \frac{2(p-1)}{p^2} \textbf{x}_5^{(1, 1)} + \frac{(3p-1)(p-1)}{p^3} \textbf{x}_5^{(1, 2)}  + \frac{(p-1)(4p - 3)}{p^3} \textbf{x}_5^{(2, 2)} \right)$\\ 
$T^{(5)}_{0, 2}$ & $x_0^5\left(  \frac{ 1}{p^2}    \textbf{x}_5^{(0, 2)} + \frac{p-1}{p^3} \textbf{x}_5^{(1, 1)} + \frac{2(p-1)}{p^3} \textbf{x}_5^{(1, 2)}  + \frac{(p-1)(3p-1) }{p^3} \textbf{x}_5^{(2, 2)} \right)$\\ \hline
$T^{(6)}_{0, 0}$ & $\begin{array}{l} x_0^6\left(\textbf{x}_6^{(0, 0)} + \frac{p-1}{p} (\textbf{x}_6^{(0, 1)}+  \textbf{x}_6^{(0, 2)} + \textbf{x}_6^{(0, 3)}) + \frac{(p-1)(2p-1)}{p^2} \textbf{x}_6^{(1, 1)} + \frac{2(p-1)^2}{p^2} (\textbf{x}_6^{(1, 2)}  + \textbf{x}_6^{(1, 3)})\right. \\
 \quad\quad\quad \left. + \frac{(p-1)(3p^2 - 3p + 1)}{p^3} \textbf{x}_6^{(2, 2)}+ \frac{(p-1)^2(  3p - 1)}{p^3} \textbf{x}_6^{(2, 3)}  + \frac{2(p-1)(2p^2 - 2p + 1)}{p^3} \textbf{x}_6^{(3,3)}\right)\end{array}$\\ 
$T^{(6)}_{0, 1}$ & $\begin{array}{l} x_0^6\left(\frac{1}{p}\textbf{x}_6^{(0, 1)} + \frac{p-1}{p^2} (  \textbf{x}_6^{(0, 2)} + \textbf{x}_6^{(0, 3)}) + \frac{2(p-1)}{p^2} \textbf{x}_6^{(1, 1)} + \frac{(p-1)(3p-1)}{p^3} \textbf{x}_6^{(1, 2)}  + \frac{(p-1)(3p-2)}{p^3}  \textbf{x}_6^{(1, 3)}\right. \\
 \quad\quad\quad \left. + \frac{4(p-1)^2}{p^3} \textbf{x}_6^{(2, 2)}+ \frac{(p-1)( 5p^2 - 4p + 1 )}{p^4} \textbf{x}_6^{(2, 3)}  + \frac{(p-1)^2(5p-1)}{p^4} \textbf{x}_6^{(3,3)}\right)\end{array}$
\end{tabular}\\[0.3cm]
\caption{Polynomials for Hecke operators under the Satake map}
\label{extension}
\end{table}

% \newpage

% \bibliography{ap_bib}
% \bibliographystyle{amsplain}
% \bibliographystyle{amsalpha}

% \end{document}

\end{document}